\documentclass[a4paper,10pt]{amsart}
\usepackage{cases}
\usepackage{graphicx}
\usepackage{amsmath}
\usepackage{amssymb}
\usepackage[all]{xypic}
\usepackage[all]{xy}
\usepackage{mathrsfs}
\usepackage{amsfonts}
\usepackage{amscd}
\usepackage{latexsym}
\usepackage{amssymb}
\usepackage{graphicx,picinpar,emlines2,epsf}

\begin{document}
\input xy
\xyoption{all}

\renewcommand{\mod}{\operatorname{mod}\nolimits}
\newcommand{\proj}{\operatorname{proj}\nolimits}
\newcommand{\rad}{\operatorname{rad}\nolimits}
\newcommand{\soc}{\operatorname{soc}\nolimits}
\newcommand{\ind}{\operatorname{ind}\nolimits}
\newcommand{\Top}{\operatorname{top}\nolimits}
\newcommand{\ann}{\operatorname{Ann}\nolimits}
\newcommand{\id}{\operatorname{id}\nolimits}
\newcommand{\Mod}{\operatorname{Mod}\nolimits}
\newcommand{\End}{\operatorname{End}\nolimits}
\newcommand{\Ob}{\operatorname{Ob}\nolimits}
\newcommand{\Ht}{\operatorname{Ht}\nolimits}
\newcommand{\cone}{\operatorname{cone}\nolimits}
\newcommand{\rep}{\operatorname{rep}\nolimits}
\newcommand{\Ext}{\operatorname{Ext}\nolimits}
\newcommand{\Hom}{\operatorname{Hom}\nolimits}
\newcommand{\Dom}{\operatorname{Dom}\nolimits}
\renewcommand{\Im}{\operatorname{Im}\nolimits}
\newcommand{\Ker}{\operatorname{Ker}\nolimits}
\newcommand{\Coker}{\operatorname{Coker}\nolimits}
\renewcommand{\dim}{\operatorname{dim}\nolimits}
\newcommand{\Iso}{\operatorname{Iso}\nolimits}
\newcommand{\Ab}{{\operatorname{Ab}\nolimits}}
\newcommand{\Coim}{{\operatorname{Coim}\nolimits}}
\newcommand{\pd}{\operatorname{pd}\nolimits}
\newcommand{\sdim}{\operatorname{sdim}\nolimits}
\newcommand{\add}{\operatorname{add}\nolimits}
\newcommand{\cc}{{\mathcal C}}
\newcommand{\co}{{\mathcal O}}
\newcommand{\cl}{{\mathcal L}}
\newcommand{\cs}{{\mathcal S}}
\newcommand{\crc}{{\mathscr C}}
\newcommand{\crd}{{\mathscr D}}
\newcommand{\crm}{{\mathscr M}}
\newcommand{\Br}{\operatorname{Br}\nolimits}
\newcommand{\coh}{\operatorname{coh}\nolimits}
\newcommand{\vect}{\operatorname{vect}\nolimits}
\newcommand{\cub}{\operatorname{cub}\nolimits}
\newcommand{\can}{\operatorname{can}\nolimits}
\newcommand{\rk}{\operatorname{rk}\nolimits}
\newcommand{\st}{[1]}
\newcommand{\X}{\mathbb{X}}
\newcommand{\N}{\mathbb{N}}
\newcommand{\ul}{\underline}
\newcommand{\vx}{\vec{x}}
\newcommand{\vc}{\vec{c}}

\newtheorem{theorem}{Theorem}[section]
\newtheorem{acknowledgement}[theorem]{Acknowledgement}
\newtheorem{algorithm}[theorem]{Algorithm}
\newtheorem{axiom}[theorem]{Axiom}
\newtheorem{case}[theorem]{Case}
\newtheorem{claim}[theorem]{Claim}
\newtheorem{conclusion}[theorem]{Conclusion}
\newtheorem{condition}[theorem]{Condition}
\newtheorem{conjecture}[theorem]{Conjecture}
\newtheorem{corollary}[theorem]{Corollary}
\newtheorem{criterion}[theorem]{Criterion}
\newtheorem{definition}[theorem]{Definition}
\newtheorem{example}[theorem]{Example}
\newtheorem{exercise}[theorem]{Exercise}
\newtheorem{lemma}[theorem]{Lemma}
\newtheorem{notation}[theorem]{Notation}
\newtheorem{proposition}[theorem]{Proposition}
\newtheorem{remark}[theorem]{Remark}
\newtheorem{solution}[theorem]{Solution}
\newtheorem{summary}[theorem]{Summary}
\newtheorem*{thma}{Theorem}
\newtheorem*{Riemann-Roch Formula}{Riemann-Roch Formula}
\newtheorem{question}[theorem]{Question}
\numberwithin{equation}{section}


\title[{Tilting bundles and the ``missing part"}]{Tilting bundles and the ``missing part" on the weighted projective line of type $(2,
2, n)$}

\author[Chen]{Jianmin Chen}
\address{School of Mathematical Sciences, Xiamen University, Xiamen 361005, P.R.China}
\email{chenjianmin@xmu.edu.cn}
\thanks{The first author was supported by the National Natural
  Science Foundation of China (Grant No. 11201386) and the Natural Science Foundation of Fujian Province of China (Grant No.
  2012J05009); the second author was supported by the National Natural
  Science Foundation of China (Grant No. 10931006 and No. 11071040); the corresponding author was supported by the National Natural
  Science Foundation of China (Grant No. 11201388)}

\author[Lin]{Yanan Lin}
\address{School of Mathematical Sciences, Xiamen University, Xiamen 361005, P.R.China}
\email{ynlin@xmu.edu.cn}

\author[Ruan]{Shiquan Ruan$^\dag$}
\address{School of Mathematical Sciences, Xiamen University, Xiamen 361005, P.R.China}
\thanks{$^\dag$ Corresponding author}
\email{shiquanruan@stu.xmu.edu.cn }

\subjclass[2000]{14F05; 16G70; 18A32; 18E10; 18E30; 18E40.}
\keywords{Tilting bundle, missing part, weighted projective line,
factor category}

\begin{abstract}
This paper classifies all the tilting bundles in the category of
coherent sheaves on the weighted projective line of weight type $(2,
2, n)$, and investigates the abelianness of the ``missing part" from
the category of coherent sheaves to the category of finitely
generated right modules on the associated tilted algebra for each
tilting bundle.
\end{abstract}

\maketitle
\section{Introduction}

Tilting theory arises from the representation theory of finite
dimensional algebras and has proved to be a universal method for
constructing equivalences between categories, (see for instance
\cite{BB, HR}). Let $A$ be a finite dimensional algebra and $_{A}T$
be a tilting left $A$-module. Happel and Ringel \cite{HR} show that
$_{A}T$ gives rise to a torsion pair $(\mathcal {T}(_{A}T); \mathcal
{F}(_{A}T))$ in the category $\mod A$ of finitely generated left
$A$-modules and a corresponding torsion pair $(\mathcal {T}(T_{B});
\mathcal {F}(T_{B}))$ in the category of $\mod B^{op}$ of finitely
generated right $B$-modules, where $B=\mbox{End}_{A}(T)$ is the
endomorphism algebra of $_{A}T$. By Brenner-Butler Theorem
\cite{BB}, the functor $\Hom_{A}(T, -)$ induces an equivalence
between the categories $\mathcal {T}(_{A}T)$ and $\mathcal
{T}(T_{B})$, and the functor $\Ext^{1}_{A}(T, -)$ induces an
equivalence between the category $\mathcal {F}(_{A}T)$ and $\mathcal
{F}(T_{B})$. Moreover, if $A$ is hereditary, the torsion pair
$(\mathcal {T}(T_{B}); \mathcal {F}(T_{B}))$ is splitting, i.e. any
indecomposable right $B$-module is either in $\mathcal {T}(T_{B})$
or in $\mathcal {F}(T_{B})$.

Geigle and Lenzing \cite{GL} extended the notion of tilting module
to tilting sheaf in the category $\coh\X$ of coherent sheaves on the
weighted projective line $\mathbb{X}$. They showed that
$T_{\can}=\bigoplus\limits_{0\leq \vx\leq \vc}\co(\vx)$ is a
(canonical) tilting sheaf in $\coh\X$, with endomorphism algebra
$\Lambda=\End_{\coh\X}(T_{\can})$ the canonical algebra of type $(2,
2, n)$. The tilting sheaf $T_{\can}$ gives rise to torsion pairs
$(\mathcal {X}_{0}, \mathcal {X}_{1}) $ in $\coh\X$ and $(\mathcal
{Y}_{1}, \mathcal {Y}_{0} )$ in $\mbox{mod}\Lambda^{op}$ by setting
$\mathcal {X}_{i}\subseteq\coh\X$ the full subcategory of all $X$
with $\Ext^{j}_{\coh\X}(T, X)=0$ for all $j\neq i$, and $\mathcal
{Y}_{i}\subseteq \mbox{mod}\Lambda^{op}$ the full subcategory of all
$Y$ with $\text{Tor}_{j}^{\Lambda}(Y, T)=0$ for all $j\neq i$.
Analogous to the main theorem of tilting theory in the module
theoretic setting due to \cite{BB} and \cite{HR}, they proved that
the functors $\Ext^{i}_{\coh\X}(T, -): \mathcal{X}_{i} \to
\mathcal{Y}_{i}$ and $\mbox{Tor}_{i}^{\Lambda}(-, T): \mathcal
{Y}_{i} \to \mathcal {X}_{i}$ define equivalences of categories,
inverse to each other, for $ i=0,1$. Moreover, the torsion pair
$(\mathcal {Y}_{1}, \mathcal {Y}_{0} )$ is splitting in $\mbox{mod}
\Lambda^{op}$.

Combining the results from \cite{GL} with further work related to
tilting theory in hereditary categories, Lenzing \cite[Theorem
3.1]{L} established the tilting theorem in a hereditary abelian
Hom-finite category $\mathscr{H}$ with Serre duality and tilting
objects. He pointed out that if the endomorphism algebra $\Lambda$
of a tilting object $T$ in $\mathscr{H}$ is not hereditary, there
must loose some objects from $\mathscr{H}$ to $\mod\Lambda^{op}$.
Notice that the category $\coh\X$ satisfies the conditions of
\cite[Theorem 3.1]{L}. It is interesting to study the tilting sheaf
in $\coh\X$ and the structure of the corresponding factor category
$\crc=\coh\X/[\mathcal {X}_{0} \cup \mathcal {X}_{1} ]$, here,
$[\mathcal {X}_{0} \cup \mathcal {X}_{1} ]$ denotes the ideal of all
morphisms in $\coh\X$ which factor through a finite direct sum of
sheaves from $\mathcal {X}_{0} \cup \mathcal {X}_{1}$.

The factor category $\crc$ is called the ``missing part" from
$\coh\X$ to $\mbox{mod}\Lambda^{op}$ in \cite{CLR}. It is only an
additive category in general. Chen, Lin and Ruan \cite{CLR} focused
on the weighted projective line of weight type $(2, 2, n)$, and
showed that for the canonical tilting sheaf $T_{\can}$, the
corresponding ``missing part" $\crc_{\can}$ is an abelian category
and isomorphic to $\mod(k \overrightarrow{A}_{n-1})$. Moreover, some
examples there indicated that the abelianness is not true if the
tilting sheaf contains a direct summand of finite length sheaf. In
this paper, we extend the result to a more general case. Namely, we
investigate the tilting sheaves not containing finite length direct
summands--called tilting bundles--in $\coh\X$ and discuss the
abelianness of the ``missing part" $\crc$ from $\coh\X$ to
$\mbox{mod}\Lambda^{op}$, where the endomorphism algebra $\Lambda$
of tilting bundle is called tilted algebra in \cite{L}. We classify
all the tilting bundles into two types, consisting of line bundles
and not all consisting of line bundles. For the former case, we show
that each tilting bundle is the canonical one (under grading shift),
hence $\crc$ is abelian; for the latter case, we give the trichotomy
of the form of each tilting bundle, and show that $\crc$ is a
product of two abelian categories.

The paper is organized as follows. In section 2, we recall the
definition of weighted projective line $\X$ of type $(2,2,n)$ and
some well-known results on the category $\coh\X$ of coherent sheaves
on $\X$. In section 3, we classify all the tilting bundles in
$\coh\X$. More precisely, if the tilting bundle $T$ is consisting of
line bundles, then $T$ can be obtained from the canonical one under
grading shift; if else, $T$ can be decomposed into three parts,
$T=T^+(E_i)\oplus(\bigoplus\limits_{i\leq k \leq j}E_k)\oplus
T^-(E_j)$, see for instance (\ref{form of tilting bundles}). In
section 4, we show that the ``missing part" corresponding to $T$ of
the form (\ref{form of tilting bundles}) can be decomposed into a
product of two abelian categories.

Throughout the paper, let $k$ be an algebraic closed field and $\X$
be a weighted projective line of weight type $(2,2,n)$ with $n
\geqslant 2$. For simplification, we denote $\Ext^{i}_{\coh\X}(-,
-)$ by $\Ext^{i}(-, -)$ for $i\geq 0$.

\section{Preliminary}

The notion of weighted projective line was introduced by Geigle and
Lenzing \cite{GL} to give a geometric treatment to canonical algebra
which was studied by Ringel \cite{R}. In this section, we introduce
basic definitions and properties on the category of coherent sheaves
on the weighted projective line of weight type $(2,2,n)$.

\subsection{Weighted projective line}

Let $p=(p_{1},p_{2},p_{3})=(2,2,n)$ and $\mathbb{L}$ be the rank one
abelian group on generators $\vx_1, \vx_2, \vx_3$ with relations
$$2\vx_1=2\vx_2=n\vx_3=:\vc.$$ Then $\mathbb{L}$ is an ordered group,
and each $\vx\in \mathbb{L}$ can be uniquely written in normal form
\begin{equation}\label{normal form of x}
\vx=\sum\limits_{i=1}^{3}l_{i}\vx_{i}+l\vc, \mbox{\quad where\quad}
0\leq l_{i}\leq p_{i}-1 \mbox{\quad and\quad} l\in \mathbb{Z}.
\end{equation}
In addition, if $\vx$ is in normal form (\ref{normal form of x}),
one can define
$$\vx\geq 0\mbox{\quad if\ and\ only\ if\quad} l\geq 0,$$ then each
$\vx\in \mathbb{L}$ satisfies exactly one of the following two
possibilities
\begin{equation}\label{two possibilities of x}
\vx\geq 0 \text{\quad or\quad}  \vx\leq \vec{\omega}+\vc,
\end{equation}
where $\vc$ is called the canonical element and
$\vec{\omega}=\vc-\sum\limits_{i=1}^{3}\vx_{i}$ is called the
dualizing element of $\mathbb{L}$.

Denote by $S$ the commutative algebra
$$S=k[X_{1},X_{2},X_{3}]/\langle X^{n}_{3}-X^{2}_{2}+
X^{2}_{1}\rangle\triangleq k[x_{1},x_{2}, x_{3}].$$ Then $S$ carries
$\mathbb{L}$-graded by setting $\deg x_{i}=\vx_i \ (i=1,2,3)$, i.e.
$S$ has a decomposition $S=\bigoplus\limits_{\vx\in
\mathbb{L}}S_{\vx}$, which satisfies $S_{\vx}S_{\vec{y}}\subseteq
S_{\vx+\vec{y}}$ and $S_{0}=k$. Let $\mathbb{X}$ be the curve
corresponding to $S$ and we call it the weighted projective line of
weight type $(2,2,n)$.

\subsection{Coherent sheaves on the weighted projective line $\X$}
In the sense of Serre-Grothendieck-Gabriel \cite{G}, the category of
coherent sheaves on the weighted projective line $\mathbb{X}$ is
defined as the quotient category
$$\coh\X=\mod^{\mathbb{L}}S/\mod_{0}^{\mathbb{L}}S,$$
where $\mod^{\mathbb{L}}S$ is the category of finitely generated
$\mathbb{L}$-graded $S$-modules and $\mbox{mod}_{0}^{\mathbb{L}}S$
the full subcategory of $\mbox{mod}^{\mathbb{L}}S$ consisting of
finite dimensional modules. Use the notation
$\widetilde{M}\in\coh\X$ for $M\in \mod^{\mathbb{L}}S$,
and call the process sheafification:
$$q:
\mod^{\mathbb{L}}S\rightarrow \coh\X;\quad M\mapsto \widetilde{M}.$$
It is easy to see that $\widetilde{M}(\vx)=\widetilde{M(\vx)}$. Call
$\mathcal{O}=\widetilde{S}$ the structure sheaf of $\mathbb{X}$, and
$\mathcal{O}(\vx)$ a line bundle for $\vx\in \mathbb{L}$. Then by
definition, for each $\vx,\vec{y}\in \mathbb{L}$, we have
\begin{equation}\label{homomorphism between line bundles}
\Hom(\co(\vx),\co(\vec{y}))=S_{\vec{y}-\vx}.
\end{equation}
In addition, the category $\coh\X$ has a decomposition:
$$\coh\X=\vect\X \bigvee
\coh_{0}\X,$$ where the full subcategory $\vect\X$ (resp.
$\coh_{0}\X$) consists of coherent sheaves not having a simple
sub-sheaf
(resp. of finite length), $\bigvee$ means each indecomposable sheaf
is either in $\vect\X$ or in $\coh_{0}\X$, and there are no non-zero
morphisms from $\coh_{0}\X$ to $\vect\X$. The objects in $\vect\X$
are called vector bundles. Moreover, all line bundles belong to
$\vect\X$, and each vector bundle $X$ has a filtration by line
bundles
$$0=X_{0}\subset X_{1}\subset X_{2}\subset\cdots\subset X_{r}=X,$$
where each factor $L_{i}=X_{i}/X_{i-1}$ is a line bundle.

Let $\Delta=[2,2,n]$ be the Dynkin diagram attached to $\X$ and
$\widetilde{\Delta}$ its extended Dynkin diagram. According to
\cite{KLM}, the Auslander-Reiten quiver $\Gamma(\vect\X)$ of
$\vect\X$ consists of a single standard component of the form
$\mathbb{Z}\widetilde{\Delta}$. Moreover, the category
$\ind(\Gamma(\vect\X))$ of indecomposable vector bundles on $\X$ is
equivalent to the mesh category of $\Gamma(\vect\X)$.

In \cite{GL}, Geigle and Lenzing showed that the category $\coh\X$
is a hereditary, abelian, $k$-linear, Hom-finite, Noetherian
category with Serre duality, i.e.
\begin{equation}\label{Serre duality}
\Ext^1(X,Y)= D\Hom(Y, \tau X),
\end{equation} where the
$k$-equivalence $\tau: \coh\X\to\coh\X$ is given by the shift
$X\mapsto X(\vec{\omega})$.

\subsection{Tilting sheaf and Grothendieck group}
Recall from \cite{GL} that a coherent sheaf $T$ is called tilting in
$\coh\X$ if the following properties hold:
\begin{itemize}
\item[(1)]$T$ is extension-free, that is, $\Ext^1(T, T)=0$;
\item[(2)]$T$ generates $D^b(\coh\X)$ as a triangulated category,
i.e. $D^b(\coh\X)$ is the smallest triangulated subcategory of
$D^b(\coh\X)$ containing $T$;
\item[(3)]$\text{gl.dim}(\End(T))<\infty.$
\end{itemize}
Geigle and Lenzing \cite{GL} showed that condition (3) is a
consequence of (1) and (2), and to prove condition (2) it is
sufficient to show that $\coh\X$ is the smallest abelian subcategory
of $\coh\X$ containing $T$ since $\coh\X$ is hereditary. Moreover,
they gave a canonical tilting sheaf $T_{\can} = \bigoplus \limits_{0
\leqslant \vx \leqslant \vc  } \mathcal {O}(\vx)$ in $\coh\X$ whose
endomorphism algebra $\Lambda=\End(T_{\can})$ is the canonical
algebra of type $(2,2,n)$.

Let $K_{0}(\X)$ be the Grothendieck group of $\coh\X$ and we still
write $X\in K_{0}(\X) $ for the class $X\in \coh\X$. Then the
classes $\mathcal{O}(\vx)$ for $0\leq \vx\leq \vc$ form a
$\mathbb{Z}$-basis of $K_{0}(\X)$. There are two additive functions
on $K_{0}(\X)$ called rank and degree respectively. The rank
function $\mbox{rk}: K_{0}(\X)\to \mathbb{Z}$ is characterized by
$$\rk(\co(\vx))=1 \text{\quad for\quad} \vx\in \mathbb{L},$$ and the
degree function $\deg: K_{0}(\X)\to\mathbb{Z}$ is uniquely
determined by setting
$$\deg(\co(\vx))=\delta(\vx)\text{\quad for\quad} \vx\in \mathbb{L},$$ where $\delta:
\mathbb{L}\to \mathbb{Z}$ is the group homomorphism defined on
generators by
$$\delta(\vx_{1})=\delta(\vx_{2})=\frac{\text{l.c.m.}(2,n)}{2} \text{\quad and\quad}
\delta(\vx_{3})=\frac{\text{l.c.m.}(2,n)}{n} .$$ For each $X\in
\coh\X$, define the slope of $X$ as $$\mu X=\frac{\deg X}{\rk X}.$$
It is an element in $\mathbb{Q} \cup \{\infty\}$. According to
\cite{GL}, each vector bundle has positive rank, then the slope
belongs to $\mathbb{Q}$; each object in $\coh_{0}\X$ has rank $0$,
then the slope is $\infty$.

In \cite{L}, Lenzing proved that each indecomposable vector bundle
$X$ is exceptional in $\coh\X$, that is, $X$ is extension-free and
$\End(X)=k$. Moreover, for any two indecomposable vector bundles $X$
and $Y$, $\mbox{Hom}(X,X')\neq 0$ implies $\mu X\leq \mu X'$.

\section{Classification of tilting bundles}
In this section, we investigate the tilting bundle in $\coh\X$. We
make discussion in two cases according to whether it is consisting
of line bundles or not. Finally we give a classification of all the
tilting bundles.

\subsection{Tilting bundle consisting of line bundles}
In \cite{GL}, we know that
\[
T_{\can}=\bigoplus\limits_{0\leq \vx\leq \vec{c}}\co(\vx)
\]
is a canonical tilting sheaf in $\coh\X$. In this subsection, we
show that it is the unique tilting bundle in $\coh\X$ consisting of
line bundles up to twist, that is, each titling bundle consisting of
line bundles has the form
\[
T_L=\bigoplus\limits_{0\leq \vx\leq \vec{c}}L(\vx), \text{\ for\
some\ line\ bundle\ } L.
\]

\begin{lemma}\label{exceptional free of line bundle}
For any line bundle $L$ and $\vx\in \mathbb{L}$, $L\oplus L(\vx)$ is
extension-free if and only if $$-\vec{c}\leq\vx\leq\vec{c}.$$ In
particular, if additional $\delta(\vx)\geq 0$, then $\vx$ satisfies
one of the following conditions:
\begin{equation}\label{three choices of x}
(i)\ 0\leq\vx\leq\vec{c}; \ \ (ii)\ \vx=\vx_1-\vx_2;\ \ (iii)\
\vx=\vx_i-k\vx_3,\text{\ for\ }i=1, 2,\ 0< k\leq\frac{n}{2}.
\end{equation}

\begin{proof}
For any $\vx, \vec{y}\in \mathbb{L}$, by Serre duality (\ref{Serre
duality}), we have
\begin{equation}\label{formula of extension between line bundles}
\Ext^1(L(\vec{y}), L(\vx))=D\Hom(L(\vx),
L(\vec{\omega}+\vec{y}))=DS_{\vec{\omega}+\vec{y}-\vx}.
\end{equation}
Hence by (\ref{two possibilities of x})
\[\Ext^1(L, L(\vx))=0 \text{\ if\ and\ only\ if\ }\vec{\omega}-\vx\leq
\vec{\omega}+\vc, \text{\ that\ is,\ } \vx\geq -\vc;
\]
and
\[\Ext^1(L(\vx),
L)=0 \text{\ if\ and\ only\ if\ }\vec{\omega}+\vx\leq
\vec{\omega}+\vc, \text{\ that\ is,\ } \vx\leq \vc.
\]
Thus, $L\oplus L(\vx)$ is extension-free if and only if
\begin{equation}\label{bound of x}
-\vec{c}\leq\vx\leq\vec{c}.
\end{equation}
In particular, if additional
\begin{equation}\label{addtional condition}\delta(\vx)\geq 0,
\end{equation}we write $\vx=\sum\limits_{i=1}^{3}l_i\vx_i+l\vc$ in
normal form and consider all the possibilities of $\vx$ satisfying
both conditions (\ref{bound of x}) and (\ref{addtional condition})
according to the number $m$ of non-zero coefficients of $\vx_i$ for
$1\leq i\leq 3$.
\begin{itemize}
\item[\emph{Case 1}:]$m=0$, then $\vx=0$ of $\vc$.
\item[\emph{Case 2}:]$m=1$, then $\vx=l_i\vx_i$ for some $1\leq i\leq 3$
and $0<l_i<p_i$, i.e., $0<\vx<\vc$.
\item[\emph{Case 3}:]$m=2$, then $\vx=\vx_1+\vx_2-\vc$ or
$\vx_i+l_3\vx_3-\vc$ for $i=1,2$ and $\frac{n}{2}\leq l_3<n$. That
is, $\vx=\vx_1-\vx_2$ or $\vx_i-k\vx_3$ for $i=1,2$ and $0<k\leq
\frac{n}{2}$.
\item[\emph{Case 4}:]$m=3$, none of $\vx$ satisfied.
\end{itemize}
Summarize up, (\ref{three choices of x}) holds. This finishes the
proof.
\end{proof}
\end{lemma}

The following is the main result in this subsection.
\begin{theorem}\label{tilting line bundle theorem}
Let $T$ be a tilting bundle in $\coh\X$ consisting of line bundles.
Then
\[
T=\bigoplus\limits_{0\leq\vx\leq\vc}L(\vx), \text{\ for\ some\ line\
bundle\ } L.
\]
\end{theorem}

\begin{proof}
Let $L$ be a direct summand of $T$ with minimal slope, and
$T=\bigoplus\limits_{\vx\in I}L(\vx)$ for some finite set $I$. By
\cite{GL}, $T_{\can}$ is a tilting sheaf in $\coh\X$ with $n+3$ many
indecomposable direct summands. It follows that the order
\begin{equation}\label{order of I}|I|=n+3.
\end{equation}
For each $\vx\in I$, by Lemma \ref{exceptional free of line bundle},
$\vx$ satisfies one of the conditions of (\ref{three choices of x}).
We claim the condition (iii) there doesn't hold. Otherwise, there
exist some $0<k\leq\frac{n}{2}$ and $i=1$ or $2$, such that $\vx_i-k
\vx_3\in I$. Assume $k_{\min}=a$, $k_{\max}=b$. Then for any $l\geq
n-b+1$, by (\ref{formula of extension between line bundles}),
$$\Ext^1(L(l\vx_3),
L(\vx_i-b\vx_3))=DS_{(l+b)\vx_3-\vx_i+\vec{\omega}}\neq 0,
$$since $(l+b)\vx_3-\vx_i+\vec{\omega}=\vx_1+\vx_2-\vx_i+(l+b-1)\vx_3-\vc\geq
\vx_1+\vx_2-\vx_i>0$. Hence $l\vx_3\notin I$.
Thus by Lemma \ref{exceptional free of line bundle}, each element
$\vx$ from $I$ satisfies one of the following:
\[(i)\ 0\leq\vx\leq (n-b)\vx_3; \quad (ii)\ \vx=\vx_1-\vx_2; \quad
(iii)\ \vx=\vx_i-k\vx_3, \ a\leq k\leq b.
\]
It follows that $$|I|\leq (n-b+1)+1+(b-a+1)=n+3-a<n+3,$$ a
contradiction to (\ref{order of I}). This finishes the claim.
Therefore,
\begin{equation}\label{thm1.1}
I\subseteq \{\vx\in\mathbb{L} | 0\leq \vx\leq \vc \text{\quad
or\quad }\vx=\vx_1-\vx_2\}.
\end{equation}
Moreover, for any $k\geq 1$,
\[
\Ext^1(L(k\vx_3), L(\vx_1-\vx_2))=DS_{(k-1)\vx_3}\neq 0.
\]
Combining with (\ref{order of I}) and (\ref{thm1.1}),
$$I=\{\vx\in\mathbb{L}|0\leq \vx\leq \vc \}.$$
That is, $$T=\bigoplus\limits_{0\leq\vx\leq\vc}L(\vx).$$ This
finishes the proof.
\end{proof}

\subsection{Tilting bundle not all consisting of line bundles}
In this section, we classify the tilting bundles not all consisting
of line bundles. Notice that if $n=2$, then all the indecomposable
direct summands of such a tilting bundle form a slice in the
Auslander-Reiten quiver of $\vect\X$, hence the classification is
obvious. Thus we only consider the case $n\geq 3$. We decompose such
a tilting bundle into three parts with respect to rank two
indecomposable direct summands of minimal and maximal length.

\subsubsection{Slice in the category of vector bundles}
According to \cite{KLM}, the Auslander-Reiten quiver
$\Gamma(\vect\X)$ of $\vect\X$ consists of a single standard
component having the form $\mathbb{Z}\widetilde{\Delta}$, where
$\widetilde{\Delta}$ is the extended Dynkin diagram with associated
Dynkin diagram $\Delta=[2,2,n]$. Moreover, the category
$\ind(\Gamma(\vect\X))$ of indecomposable vector bundles on $\X$ is
equivalent to the mesh category of $\Gamma(\vect\X)$. Furthermore,
each indecomposable vector bundle has rank one or two.

For each indecomposable vector bundle $X$, denote by $\cs(X\to)$
(resp. $\mathcal{S}(\to X)$) the slice beginning from $X$ (resp.
ending to $X$) in $\Gamma(\vect\X)$, for the definition of slice we
refer to \cite{R}. More precisely,
\[
\cs(X\to)=\{Y\in\ind(\vect\X)|\Hom(X, Y)\neq 0 \text{\ and\ }\Hom(X,
\tau^m Y)=0  \text{\ for\ }m\geq 1\}, \]and
\[
\cs(\to X)=\{Y\in\ind(\vect\X)|\Hom(Y, X)\neq 0 \text{\ and\
}\Hom(\tau^{-m}Y,  X)=0 \text{\ for\ }m\geq 1\}. \]

The following lemma plays an important role in classifying the
tilting bundles which are not all consisting of line bundles in
$\coh\X$.
\begin{lemma}\label{homomorphisms related to rank two bundle}Let $E$ and $X$ be
two indecomposable vector bundles with $\rk E=2$.
\begin{itemize}
\item[(1)] If $\Hom(E, X)\neq 0$, then $\Hom(E, \tau^{-1}X)\neq 0$.
\item[(2)] If $\Hom(X, E)\neq 0$, then $\Hom(\tau X, E)\neq 0$.
\end{itemize}
\end{lemma}

\begin{proof}It suffices to prove statement (1), since the
arguments for statement (2) are dual. For contradiction, we assume
$X$ is of minimal slope satisfying
\[\Hom(E, X)\neq 0 \text{\ but\ } \Hom(E, \tau^{-1}X)= 0.
\]
Concerning the rank of $X$, we consider the following two cases.
\begin{itemize}
\item[\emph{Case 1:}]$\rk X=1$, then $X$ is a line bundle $L$. In the following sub-quiver of $\Gamma(\vect\X)$,
set $L'=L(\vx_1-\vx_2)$, and $E_L$ the Auslander bundle
corresponding to $L$.
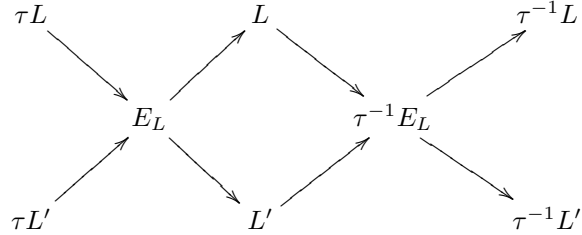
$$
\xymatrix{
  \tau L \ar[rd] & & L \ar[rd] &  & \tau^{-1}L \\
& E_{L} \ar[rd] \ar[ru] &  &\tau^{-1}E_{L} \ar[rd] \ar[ru] &  \\
  \tau L' \ar[ru] & & L' \ar[ru]&  &\tau^{-1}L'}
  $$
We claim that
\begin{equation}\label{claim of (1)}
\Hom(E, L')\neq 0.
\end{equation}
In fact, applying $\Hom(E, -)$ to the Auslander-Reiten sequence
\[0\to\tau L\to E_L\to L\to 0,
\]
we obtain an exact sequence:
\begin{equation}\label{for contradiction (1)1}
\Hom(E, E_L)\to\Hom(E, L)\to\Ext^1(E, \tau L).
\end{equation}
By assumption, $\Hom(E, L)\neq 0$, which implies $\mu E<\mu L$. It
follows that $\Ext^1(E, \tau L)=D\Hom(L, E)=0$. Hence
\begin{equation}\label{for contradiction (1)2}
\Hom(E, E_L)\neq 0.
\end{equation}
Now by applying $\Hom(E, -)$ to the Auslander-Reiten sequence
\[0\to\tau L'\to E_L\to L'\to 0,
\]
we obtain an exact sequence:
\[
0\to\Hom(E, \tau L')\to\Hom(E, E_L)\to\Hom(E, L').
\]
If $\Hom(E, \tau L')= 0$, then (\ref{claim of (1)}) follows from
(\ref{for contradiction (1)2}); if $\Hom(E, \tau L')\neq 0$, then by
the minimality of $L$, (\ref{claim of (1)}) also holds; this
finishes the claim. Then by applying $\Hom(E, -)$ to the injective
map $L'\rightarrowtail L'(\vx_3)$, it follows from (\ref{claim of
(1)}) that $\Hom(E, \tau^{-1}L)=\Hom(E, L'(\vx_3))\neq 0$, a
contradiction.
\item[\emph{Case 2:}]
$\rk X=2$, then there exists the following Auslander-Reiten sequence
\[0\to X\to \oplus_i Y_i\to \tau^{-1}X\to 0.
\]
Since $\Hom(E, X)\neq 0$, there exists some $i$, such that $\Hom(E,
Y_i)\neq 0$. Moreover, $\rk Y_i\leq 2=\rk (\tau^{-1}X)$ implies the
irreducible map $Y_i\to \tau^{-1}X$ is injective. It follows that
$\Hom(E, \tau^{-1}X)\neq 0$, a contradiction.
\end{itemize}
\end{proof}

More general, we have
\begin{corollary}\label{general homomorphisms related to rank two bundle}Let $E$ and $X$ be two
indecomposable vector bundles with $\rk E=2$. Then for any $m\geq
0$,
\begin{itemize}
\item[(1)] $\Hom(E, X)\neq 0$ implies $\Hom(E, \tau^{-m}X)\neq 0$;
\item[(2)] $\Hom(X, E)\neq 0$ implies $\Hom(\tau^m X, E)\neq 0$.
\end{itemize}
\end{corollary}

\begin{remark}
As a consequence, the expression of the slices with respect to rank
two bundle $E$ can be simplified as follows.
\begin{equation}\label{definition of slice from X}
\cs(E\to)=\{Y\in\ind(\vect\X)|\Hom(E, Y)\neq 0 \text{\ and\ }\Hom(E,
\tau Y)=0\},
\end{equation} and
\begin{equation}\label{definition of slice to X}
\cs(\to E)=\{Y\in\ind(\vect\X)|\Hom(Y, E)\neq 0 \text{\ and\
}\Hom(\tau^{-1}Y,  E)=0\} .
\end{equation}
\end{remark}

By considering the tilting bundle containing rank two indecomposable
direct summand, we have the following important observation:
\begin{lemma}\label{at most one member}Let $E$ be an
indecomposable vector bundle with $\rk E=2$. Then each tilting
object in $\coh\X$ contains at most one member from $\tau$-orbit of
$E$.
\end{lemma}

\begin{proof}For contradiction, we assume there exists a tilting sheaf
in $\coh\X$ containing $\tau^{m_1}E$ together with $\tau^{m_2}E$ for
some $m_1<m_2$. Then by Corollary \ref{general homomorphisms related
to rank two bundle},
\[\Ext^1(\tau^{m_1}E, \tau^{m_2}E)=D\Hom(E, \tau^{m_1-m_2+1}E)\neq 0.
\]Hence $\tau^{m_1}E\oplus\tau^{m_2}E$ is not extension-free, a
contradiction.
\end{proof}

\subsubsection{Domains and properties}

Assume $\mathcal{S}(\co\to)$ in $\Gamma(\vect\X)$
is given by:
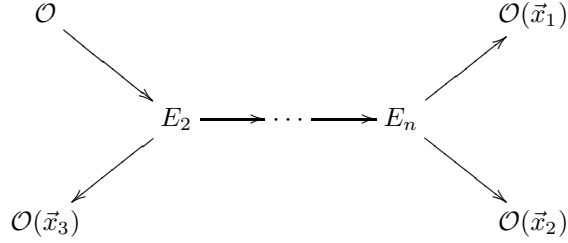
\[
\xymatrix{
  \co \ar[dr] & &&& \co(\vx_1)\\
 & E_2 \ar[ld] \ar[r] & \cdots \ar[r] &E_n \ar[ru]\ar[rd] &  \\
  \co(\vx_3) &&&& \co(\vx_2)  }
\]
Denote by $\cl_i$ the $\tau$-orbit of $E_i$, for $2\leq i\leq n$,
and by $\cl_1$ (resp. $\cl'_1, \cl_{n+1}, \cl'_{n+1}$) the
$\tau$-orbit of $\co$ (resp. $\co(\vx_3)$, $\co(\vx_1)$,
$\co(\vx_2)$).

\begin{definition}\label{definition of orbits}
For each indecomposable vector bundle $X$, define the \emph{length}
of $X$ by
\[l(X)=\begin{cases}
\begin{array}{cl}
1, &\text{\ if\ } X\in \mathcal{L}_1 \text{\ or\ } \mathcal{L}'_1;\\
i, &\text{\ if\ } X\in \mathcal{L}_i \text{\ for\ } 2\leq i\leq n;\\
n+1, &\text{\ if\ } X\in \mathcal{L}_{n+1} \text{\ or\ }
\mathcal{L}'_{n+1}.
\end{array}
\end{cases}
\]
\end{definition}

\begin{definition}Let $X$ be an indecomposable vector bundle,
the \emph{domain} of $X$, denoted by $\Dom(X)$, is defined as the
subset of $\vect\X$ consisting of all indecomposable vector bundle
$Y$ satisfying that there exist integers $ m_1,m_2\geq 0$, such that
$\tau^{m_1} Y\in\cs(\to X)$ and $\tau^{-m_2}Y\in\cs(X\to)$. In
particular, denote by
\[\Dom^+(X)=\{Y\in\Dom(X)|l(Y)\leq l(X)\}
\]
and
\[\Dom^-(X)=\{Y\in\Dom(X)|l(Y)\geq l(X)\}.
\]
\end{definition}


\begin{example}Let $\X$ be the weighted projective line of weight type
$(2,2,3)$. The following is a sub-quiver of $\Gamma(\vect\X)$.
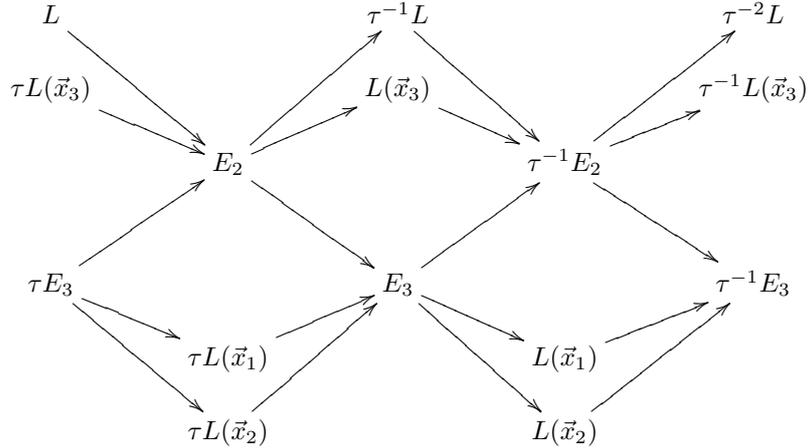
$$\xymatrix@-1pc{
  L \ar[drrd] &&&&\tau^{-1}L \ar[ddrr]  &&&& \tau^{-2}L\\
  \tau L(\vx_3) \ar[rrd] &&&& L(\vx_3) \ar[rrd]  &&&& \tau^{-1}L(\vx_3) \\
   && E_2 \ar[rrdd]\ar[rru]\ar[rruu] &&&& \tau^{-1}E_2 \ar[rrdd]\ar[rru]\ar[rruu] \\
   \\
  \tau E_3 \ar[rrdd]\ar[rrd]\ar[rruu]&&&& E_3 \ar[rrdd]\ar[rrd]\ar[rruu]&&&& \tau^{-1}E_3\\
  && \tau L(\vx_1) \ar[rru]  &&&& L(\vx_1) \ar[rru]  \\
  && \tau L(\vx_2) \ar[rruu]   &&&& L(\vx_2) \ar[rruu]    }
$$
In this picture, we have
$$\begin{array}{l}
\cs(L\to)=\{L, L(\vx_1),
L(\vx_2), L(\vx_3), E_2, E_3\}; \\[0.5em]
\cs(E_3\to)=\{E_3, \tau^{-1}E_2, L(\vx_1),
L(\vx_2),\tau^{-1}L(\vx_3), \tau^{-2}L\}; \\[0.5em]
\cs(\to E_3)=\{L, \tau L(\vx_1),\tau L(\vx_2),\tau L(\vx_3), E_2,
E_3\};
\end{array}$$
and $\Dom(E_3)$ has the form below.
$$
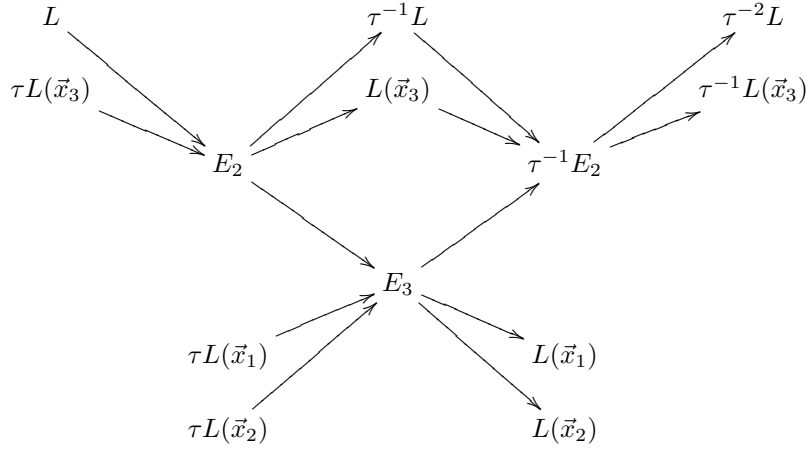
\xymatrix@-1pc{
  L \ar[drrd] &&&&\tau^{-1}L \ar[ddrr]  &&&& \tau^{-2}L\\
  \tau L(\vx_3) \ar[rrd] &&&& L(\vx_3) \ar[rrd]  &&&& \tau^{-1}L(\vx_3) \\
   && E_2 \ar[rrdd]\ar[rru]\ar[rruu] &&&& \tau^{-1}E_2 \ar[rru]\ar[rruu] \\
   \\
 &&&& E_3 \ar[rrdd]\ar[rrd]\ar[rruu]&&&&\\
  && \tau L(\vx_1) \ar[rru]  &&&& L(\vx_1)   \\
  && \tau L(\vx_2) \ar[rruu]   &&&& L(\vx_2)     }
$$
Moreover, $\Dom^+(E_3)$ (resp. $\Dom^-(E_3)$) consists of all the
bundles posit on the above (resp. below) of $E_3$, in each case,
containing $E_3$.

\end{example}

\begin{lemma}\label{Dom iff}Let $X$ and $Y$ be two indecomposable vector bundles. Then
\[Y\in\Dom(X) \text{\quad if\ and\ only\ if\quad }X\in\Dom(Y).
\]
\end{lemma}

\begin{proof}
By definition, $Y\in\Dom(X)$ if and only if there exist $m_1,m_2\geq
0$, such that $\tau^{m_1} Y\in\cs(\to X)$ and
$\tau^{-m_2}Y\in\cs(X\to)$, equivalently, if and only if $X\in
\cs(\tau^{m_1} Y\to)$ and $X\in\cs(\to\tau^{-m_2}Y)$, thus if and
only if $\tau^{-m_1} X\in \cs(Y\to)$ and $\tau^{m_2}X\in\cs(\to Y)$,
that is, $X\in\Dom(Y)$.
\end{proof}

\begin{proposition}\label{proposition of domain}
For any indecomposable vector bundles $X$ and $Y$, the following
statements are equivalent.
\begin{itemize}
\item[(1)]$X\in\Dom^+(Y)$;
\item[(2)]$\Dom^+(X)\subseteq\Dom^+(Y)$;
\item[(3)]$Y\in\Dom^-(X)$;
\item[(4)]$\Dom^-(Y)\subseteq\Dom^-(X)$.
\end{itemize}
\end{proposition}

\begin{proof}
Firstly, we show that statements (1) and (2) are equivalent. In
fact, if (2) holds, then (1) follows from $X\in\Dom^+(X)$. On the
other hand, assume (1) holds, then by definition, there exists some
$m'_{1}\geq 0$, such that $\tau^{m'_{1}}X\in\cs(\to Y)$. For any
indecomposable vector bundle $Z\in\Dom^+(X)$, there exists some
$m^{''}_{1}\geq 0$, such that $\tau^{m^{''}_{1}} Z\in\cs(\to X)$.
Let $m_1= m'_1+m^{''}_{1}$. Then $m_1\geq 0$ and
$\tau^{m_1}Z\in\cs(\to Y)$. Similarly, there exists some $m_2\geq
0$, such that $\tau^{-m_2}Z\in\cs(\to Y)$. Hence $Z\in\Dom^+(Y)$,
that is, (2) holds, as claimed.

By dually, we have statements (3) and (4) are equivalent. Moreover,
by definition, $X\in\Dom^+(Y)$ if and only if $X\in\Dom(Y)$ and
$l(X)\leq l(Y)$, thus if and only if $Y\in\Dom^-(X)$ by Lemma
\ref{Dom iff}. That is, statements (1) and (3) are equivalent, this
finishes the proof.
\end{proof}

For rank two indecomposable vector bundles, there is an equivalent
description of their domains, related to extension-free and then
tilting objects.

\begin{lemma}\label{iff condition of extension free}Let $E$ and $X$ be
two indecomposable vector bundles with $\rk E=2$. Then $X\in\Dom(E)$
if and only if $E\oplus X$ is extension-free.
\end{lemma}

\begin{proof}If $X\in\Dom(E)$, there exist $m_1, m_2\geq 0$, such
that $$\tau^{m_1} X\in\cs(\to E) \text{\ and\ }
\tau^{-m_2}X\in\cs(E\to).$$ By (\ref{definition of slice from X})
and (\ref{definition of slice to X}), we have \[\Hom(\tau^{m_1-1} X,
E)=0 \text{\ and\ } \Hom(E, \tau^{-m_2+1} X)=0.
\]
Then by Corollary \ref{general homomorphisms related to rank two
bundle},
\[\Hom(\tau^{-1} X, E)=0 \text{\ and\ } \Hom(E,
\tau X)=0.
\]
Using Serre duality, we get
\[\Ext^1(E, X)=0 \text{\ and\ } \Ext^1(X,
E)=0.
\]
That is, $E\oplus X$ is extension-free.

The sufficiency follows by similar considerations, by going the
steps of the preceding proof backwards.
\end{proof}

More general, we have
\begin{lemma}\label{extension-free of domains}
Let $E$ be an indecomposable vector bundle of rank two. Then for any
indecomposable objects $X\in\Dom^+(E)$ and $Y\in\Dom^-(E)$, $X\oplus
Y$ is extension-free.
\end{lemma}

\begin{proof}
By Proposition \ref{proposition of domain}, $X\in\Dom^+(E)$ implies
$E\in\Dom^-(X)$, and then from $Y\in\Dom^-(E)$ we get
$Y\in\Dom^-(X)$, equivalently, $X\in\Dom^+(Y)$. Hence, if $\rk X=2$
or $\rk Y=2$, then by Lemma \ref{iff condition of extension free},
$X\oplus Y$ is extension-free. If else, $X, Y$ are both line
bundles. Then from the structure of $\Dom(E)$ and by Lemma
\ref{exceptional free of line bundle}, it's easy to see that
$X\oplus Y$ is also extension-free, as claimed.
\end{proof}

\subsubsection{Classification Theorem}
Now we will give our main result of this section. Before giving the
classification theorem, we still need some preparations.

\begin{lemma}\label{bound of order I}Let $E$ be an indecomposable
vector bundle of rank two with $\cs(\to E)\cap\cl_1=L$, and
$F=\bigoplus\limits_{\vx\in I}L(\vx)$ be a direct sum of pair-wise
distinct line bundles from $\Dom^+(E)$. If $E\oplus F$ is
extension-free, then the order $|I|\leq l(E)$.
\end{lemma}

\begin{proof}From the structure of $\Gamma(\vect\X)$, we
know that each line bundle from $\Dom^+(E)$ has the form
\[ L(k\vx_3+\vx_1-\vx_2) \text{\quad or\quad } L(k\vx_3), \text{\quad for\quad }
0\leq k\leq l(E)-1.
\]
By symmetry of $\cl_1$ and $\cl'_1$, without loss of generality, we
assume there exists some $0\leq a\leq l(E)-1$, such that $L(a\vx_3)$
is a direct summand of $F$ with minimal slope. Then by Lemma
\ref{exceptional free of line bundle}, for each $\vx\in I$,
\begin{equation}\label{two choices of line bundle}
\vx=a\vx_3+\vx_1-\vx_2 \text{\quad or\quad} \vx=k\vx_3, \text{\quad
for\quad} a\leq k\leq l(E)-1.
\end{equation}
Notice from (\ref{formula of extension between line bundles}) that
$$\Ext^1(L(k\vx_3), L(a\vx_3+\vx_1-\vx_2))=DS_{(k-a-1)\vx_3}.$$
Hence, for any $k>a$, $L(k\vx_3)\oplus L(a\vx_3+\vx_1-\vx_2)$ is not
extension-free. Combining with (\ref{two choices of line bundle}),
we get $|I|\leq \max\{2,l(E)-a\}\leq l(E)$.
\end{proof}

\begin{remark}Keep the notation in Lemma \ref{bound of order I}. Denote by
\[E^u=\bigoplus\limits_{0\leq k\leq
l(E)-1}L(k\vx_3).
\]
In particular, if $l(E)=2$, denote by
\[
E^{l}=L\oplus L(\vx_1-\vx_2).
\]
Now define a set $\Br^+(E)$ as below: if $l(E)>2$, then
$\Br^+(E)=\{E^{u}, E^u(\vx_1-\vx_2)\}$; if $l(E)=2$, then
$\Br^+(E)=\{E^{l}, E^{l}(\vx_3), E^{u}, E^u(\vx_1-\vx_2)\}$. Then
according to the proof of Lemma \ref{bound of order I}, we have
\begin{equation}\label{when the order of I is maximal}
|I|=l(E) \text{\quad if\ and\ only\ if\quad} F\in\Br^+(E).
\end{equation}
\end{remark}

Dually, we have the follow lemma and remark.

\begin{lemma}\label{bound of order I2}Let $E$ be an indecomposable
vector bundle of rank two with $\cs(E\to)\cap\cl_{n+1}=L'$, and
$F=\bigoplus\limits_{\vx\in J}L'(\vx)$ be a direct sum of pair-wise
distinct line bundles from $\Dom^-(E)$. If $E\oplus F$ is
extension-free, then $|J|\leq n-l(E)+2$.
\end{lemma}

\begin{remark}Keep the notation in Lemma \ref{bound of order I2}. Denote by
\[E^d=\bigoplus\limits_{0\leq k\leq
n-l(E)+1}L'(-k\vx_3).
\]
In particular, if $l(E)=n$, denote by
\[
E^{r}=L'\oplus L'(\vx_1-\vx_2).
\]
Now define a set $\Br^-(E)$ as below: if $l(E)<n$, then
$\Br^-(E)=\{E^{d}, E^{d}(\vx_1-\vx_2)\}$; if $l(E)=n$, then
$\Br^-(E)=\{E^{r}, E^{r}(-\vx_3), E^{d}, E^{d}(\vx_1-\vx_2)\}$. Then
in Lemma \ref{bound of order I2}, we have
\begin{equation}\label{when the order of J is maximal}
|J|= n-l(E)+2 \text{\quad if\ and\ only\ if\quad} F\in\Br^-(E).
\end{equation}
\end{remark}

For convenience to describe the classification theorem, we need to
introduce a conception named sub-slice.
\begin{definition}\label{sub-slice}Assume $2\leq i\leq j\leq n$. Let $E_k\in \cl_k$ be
indecomposable vector bundles for $i\leq k\leq j$. Then $\{E_k|i\leq
k\leq j\}$ is called a \emph{sub-slice}  from $E_i$ to $E_j$ if it
can be extended to a slice in $\Gamma(\vect\X)$.
\end{definition}




\begin{lemma}\label{generate from E(i) to E(i+1)}
Assume $L\in\cl_1$, and $\{E_k|E_k\in\cl_k, 2\leq k\leq n\}$ is a
sub-slice from $E_2$ to $E_n$ contained in the slice $\cs(L\to)$.
Then for $2\leq k\leq n-1$,
\begin{itemize}
\item[(1)]$E_{k+1}^u=E_k^u\oplus L(k\vx_3)$;
\item[(2)]$E_k\in\langle E_{k+1}, L(k\vx_3),
L((k-1)\vx_3)\rangle$, the smallest full subcategory of $\coh\X$
containing $E_{k+1}, L(k\vx_3)\text{\ and \ } L((k-1)\vx_3)$ closed
under the third term of exact sequence.
\end{itemize}
\end{lemma}

\begin{proof}Assertion (1) directly follows from the definition of $E_k^u$.
For assertion (2), we consider the following exact sequence obtained
by induction on $k$: $$0\to L\to E_k\to
L(-\vec{\omega}+(k-2)\vx_3)\to 0,$$ which induces a pullback
commutative diagram as follows:
\[\xymatrix{
  0  \ar[r] & L \ar@{=}[d] \ar[r] & E_k \ar[d]\ar@{}[dr]|-{\circlearrowleft}  \ar[r] & L(-\vec{\omega}+(k-2)\vx_3) \ar[d] \ar[r] & 0 \\
  0 \ar[r]  & L \ar[r]  & E_{k+1} \ar[r]  & L(-\vec{\omega}+(k-1)\vx_3)  \ar[r] & 0.  }\]
Then we obtain an exact sequence:
\begin{equation}\label{exact sequence determine E(k)}
0\to E_k \to
E_{k+1}\to S_{3,k}\to 0,
\end{equation}
where $S_{3,k}$ is a simple sheaf concentrated at the point $\vx_3$
determined by the following exact sequence:
\begin{equation}\label{exact sequence determinate the simple sheaf}
0\to L(-\vec{\omega}+(k-2)\vx_3) \to
L(-\vec{\omega}+(k-1)\vx_3)\to S_{3,k}\to 0.
\end{equation}
Notice that $S_{3,k}(\vx_1-\vx_2)=S_{3,k}$ and
$-\vec{\omega}=-\vx_1+\vx_2+\vx_3$. We obtain the following exact
sequence obtained from (\ref{exact sequence determinate the simple
sheaf}) by taking grading shift of $\vx_1-\vx_2$:
\begin{equation}\label{exact sequence determinate the simple sheaf2}
0\to L((k-1)\vx_3) \to L(k \vx_3)\to S_{3,k}\to 0.
\end{equation}
Combining with (\ref{exact sequence determine E(k)}) and (\ref{exact
sequence determinate the simple sheaf2}), we finish the proof.
\end{proof}

By duality, we have
\begin{lemma}\label{generate from E(i) to E(i+1)2}
Assume $L'\in\cl_{n+1}$, and $\{E_k|E_k\in\cl_k, 2\leq k\leq n\}$ is
a sub-slice from $E_2$ to $E_n$ contained in the slice $\cs(\to
L')$. Then for $2\leq k\leq n-1$,
\begin{itemize}
\item[(1)]$E_k^d=E_{k+1}^d\oplus L'(-(n-k+1)\vx_3)$;
\item[(2)]$E_{k+1}\in\langle E_k, L'(-(n-k+1)\vx_3),
L'(-(n-k)\vx_3)\rangle$.
\end{itemize}
\end{lemma}

Now we give the main result in this section.
\begin{theorem}[Classification theorem]\label{classification theorem} Assume $T$ is
a bundle in $\coh\X$ not all consisting of line bundles. Then $T$ is
tilting if and only if there exist $2\leq i\leq j\leq n$ and
$E_k\in\cl_k$ for $i\leq k\leq j$, such that
\begin{equation}\label{form of tilting bundles}
T=T^+(E_i)\oplus(\bigoplus\limits_{i\leq k \leq j}E_k)\oplus
T^-(E_j), \end{equation}
where $T^+(E_i)\in \Br^+(E_i)$, $T^-(E_j)\in
\Br^-(E_j)$ and $\{E_k|i\leq k\leq j\}$ is a sub-slice  from $E_i$
to $E_j$.
\end{theorem}

\begin{proof}On one hand, we show that each tilting bundle $T$ in
$\coh\X$ not all consisting of line bundles has the form (\ref{form
of tilting bundles}). Let $E_i$ (resp. $E_j$) be the rank two
indecomposable direct summand of $T$ with minimal (resp. maximal)
length. Then by Corollary \ref{at most one member}, $E_i$ (resp.
$E_j$) is uniquely determined. Moreover, $T$ has a decomposition
$$T=T_1\oplus T_2\oplus T_3,$$ where the indecomposable
summands of $T_1$ (resp. $T_3$) are of length 1 (resp. $n+1$), and
$T_2\in\bigoplus\limits_{i\leq k\leq j}\cl_k$. Then by Lemma
\ref{iff condition of extension free}, $T_1\oplus E_i$ is
extension-free implies that $T_1\in\Dom^+(E_i)$, and $T_3\oplus E_j$
is extension-free implies that $T_3\in\Dom^-(E_j)$. Thus by Lemma
\ref{bound of order I} and \ref{bound of order I2}, we have
\begin{equation}\label{inequality1}|T_1|\leq i \text{\quad and\quad}
|T_3|\leq n-j+2,
\end{equation}
where $|T_m|$ denotes the number of pair-wise distinct
indecomposable direct summands of $T_m$. Moreover, Lemma \ref{at
most one member} implies
\begin{equation}\label{inequality2}|T_2|\leq j-i+1.
\end{equation}
It follows that $|T|=\sum\limits_{i=1}^3|T_i|\leq n+3=|T|$. Hence
each inequality in (\ref{inequality1}) and (\ref{inequality2})
should be equality. Thus $T_1\in \Br^+(E_i)$ (resp. $T_3\in
\Br^-(E_j)$) by (\ref{when the order of I is maximal}) (resp.
(\ref{when the order of J is maximal})), and
$T_2=\bigoplus\limits_{i\leq k\leq j}E_k$ with $E_k\in\cl_k$ for any
$i\leq k\leq j$. Furthermore, for $i\leq k<j$, $E_k\oplus E_{k+1}$
is extension-free implies that $E_k\in\Dom^+(E_{k+1})$. By the
structure of $\Gamma(\vect\X)$, there exists an irreducible map
between $E_k$ and $E_{k+1}$. Hence $\{E_k|i\leq k\leq j\}$ is a
sub-slice  from $E_i$ to $E_j$.

On the other hand, we show that if $T$ has the form (\ref{form of
tilting bundles}), then $T$ is a tilting sheaf in $\coh\X$. We only
consider the case
\begin{equation}\label{simple form of T}
T=E_{i}^{u}\oplus(\bigoplus\limits_{i\leq k \leq j}E_k)\oplus
E_{j}^{d},
\end{equation} since the proofs for the other choices of
$T^+(E_i)\in \Br^+(E_i)$ and $T^-(E_j)\in \Br^-(E_j)$ are quite
similar.

Firstly, we claim that $T$ of the form (\ref{simple form of T}) is
extension-free. In fact, since $\{E_k|i\leq k\leq j\}$ is a
sub-slice from $E_i$ to $E_j$, we know that $\bigoplus\limits_{i\leq
k \leq j}E_k$ is a direct summand of a tilting sheaf corresponding
to a slice, hence it is extension-free. Meanwhile, Lemma
\ref{exceptional free of line bundle} implies that both of
$E_{i}^{u}$ and $E_{j}^{d}$ are extension-free. Moreover, notice
that $E_{i}^{u}\in\Dom^+(E_i)$ and
$E_{j}^{d}\in\Dom^-(E_j)\subseteq\Dom^-(E_i)$, hence
$E_{i}^{u}\oplus E_{j}^{d}$ is extension-free by Lemma
\ref{extension-free of domains}. Furthermore, for each $i\leq k \leq
j$, $E_{i}^{u}\in\Dom^+(E_i)\subseteq\Dom^+(E_k)$ implies
$E_{k}\oplus E_{i}^{u}$ is extension-free, and
$E_{j}^{d}\in\Dom^-(E_j)\subseteq\Dom^-(E_k)$ implies $E_{k}\oplus
E_{j}^{d}$ is extension-free, as claimed.

Secondly, we remain to prove that $\coh\X$ is generated by $T$. We
extend the sub-slice $\{E_k|i\leq k\leq j\}$ to a slice as
following,
\[
\xymatrix@-1pc{
  L \ar[dr] &&&&&&&&&& L'\\
 & E_2 \ar[ld] \ar[r] & \cdots \ar[r] &E_i \ar@{-}[r]& \cdots \ar@{-}[r]&E_k \ar@{-}[r]& \cdots \ar@{-}[r]&E_j \ar[r] & \cdots \ar[r]   &E_n \ar[ru] &  \\
  L(\vx_3) &&&&&&&&&& L'(-\vx_3)\ar[lu]  }
\]where $E_m\in\cl_{m}$ for $2\leq m\leq n$, $L=\cs(\to E_i)\cap \cl_1$, $L'=\cs(E_j \to)\cap \cl_{n+1}$,
$E_i - \cdots-E_k -\cdots- E_j$ represents the given sub-slice from
$E_i$ to $E_j$ and each arrow $E_m\to E_{m+1}$ represents an
irreducible injective for $m<i$ and $m\geq j$. Then the direct sum
of all the bundles from the slice forms a tilting bundle $T'$ in
$\coh\X$. By Lemma \ref{generate from E(i) to E(i+1)}, for any
$m<i$, $E_m\in\langle E_{m+1}, L(m\vx_3), L((m-1)\vx_3)\rangle$. It
follows that $\bigoplus\limits_{m=2}^{i-1}E_m\in\langle E_i,
\bigoplus\limits_{a=1}^{i-1}L(a\vx_3) \rangle\subseteq\langle
E_i\oplus E_i^u\rangle$. Similarly, by Lemma \ref{generate from E(i)
to E(i+1)2}, we can get $\bigoplus\limits_{m=j+1}^{n}E_m\in\langle
E_j\oplus E_j^d\rangle$. Hence $\coh\X=\langle T'\rangle\subseteq
\langle T \rangle \subseteq \coh\X$, as claimed.
\end{proof}

\section{The structure of the ``missing part"}

In this section we investigate the ``missing part" from $\coh\X$ to
$\mod\Lambda^{op}$ for any tilting bundle $T$, where
$\Lambda=\End(T)$ is the endomorphism algebra of $T$.

Firstly, we recall the definition of ``missing part" from
\cite{CLR}. Let $T$ be a tilting bundle in $\coh\X$ with
endomorphism algebra $\Lambda$. Then $T$ gives rise to torsion pairs
$(\mathcal {X}_{0}, \mathcal {X}_{1}) $ in $\coh\X$ and $(\mathcal
{Y}_{1}, \mathcal {Y}_{0} )$ in $\mod\Lambda^{op}$ by setting
$$ \mathcal {X}_{i}=\{X\in\coh\X|\Ext^{j}(T, X)=0, j\neq i\}
$$
and
$$\mathcal{Y}_{i}=\{Y\in\mod\Lambda^{op}|\mbox{Tor}_{j}^{\Lambda}(Y, T) =0,j\neq i\}.
$$ Geigle and Lenzing \cite{GL} proved that the functors $$
\Ext^i(T, -): \mathcal {X}_{i} \to \mathcal {Y}_{i} \text{\quad
and\quad} \mbox{Tor}_{i}^{\Lambda} ( -,T): \mathcal {Y}_{i} \to
\mathcal {X}_{i}  $$ define equivalences of categories, inverse to
each other, for $ i=0,1$. In particular, the torsion pair $(\mathcal
{Y}_{1}, \mathcal {Y}_{0} )$ is splitting. Hence it's natural to get
the following definition:

\begin{definition}[\cite{CLR}]\label{definition of missing part }
Let $T$ be a tilting bundle in $\coh\X$ with endomorphism algebra
$\Lambda$. The ``missing part" from $\coh\X$ to $\mod(\Lambda^{op})$
is defined to be the factor category
$$\mathscr{C}=\coh(\mathbb{X})/[\mathcal {X}_{0} \cup \mathcal
{X}_{1} ],$$ where $[\mathcal {X}_{0} \cup \mathcal {X}_{1} ]$
denotes the ideal of all morphisms in $\coh\X$ which factor through
a finite direct sum of coherent sheaves from $\mathcal {X}_{0} \cup
\mathcal {X}_{1}$.
\end{definition}

By definition, the ``missing part" $\crc$ is an additive category
and has the expression
\begin{equation}\label{form of missing part}
\mathscr{C}=\add\{X\in\ind(\coh\X)|\Hom(T, X)\neq 0\neq \Ext^{1}(T,
X)\}.
\end{equation}
Notice that if $T$ is consisting of line bundles, then by Theorem
\ref{tilting line bundle theorem}, $T$ has the form
$$T_{L}=\bigoplus\limits_{0\leq\vx\leq\vc}L(\vx), \text{\quad for\ some\ line\ bundle\ } L.$$
Hence, by \cite{CLR}, the corresponding ``missing part" $\crc_{L}$
has the expression
\begin{equation}\label{form of canonical missing part}
\crc_{L}=\add\{X\in\ind(\coh\X)|\Hom(L, X)\neq 0\neq
\Ext^{1}(L(\vc), X)\},
\end{equation}
which is an abelian category. Therefore, we only consider the case
$T$ not all consisting of line bundles. In particular, for $n=2$,
the indecomposable direct summands of such a tilting bundle form a
slice in $\Gamma(\vect\X)$, whose endomorphism algebra is
hereditary, hence nothing is missing from $\coh\X$ to
$\mod\Lambda^{op}$, that is, $\crc=0$. Thus we deduce to the case
$n\geq 3$. According to Theorem \ref{main theorem}, $T$ has the form
(\ref{form of tilting bundles}). In order to describe the details
more precisely, we only consider $T$ of the form
\begin{equation}\label{form of T}
T=E_i^u\oplus(\bigoplus\limits_{i\leq k \leq j}E_k)\oplus E_j^{d}.
\end{equation}
The arguments for other choices of $T^+(E_i)\in \Br^+(E_i)$ and
$T^-(E_j)\in \Br^-(E_j)$ are quite similar.

The following lemma is crucial in this section, which is useful to
give a more explicit description of $\crc$.
\begin{lemma}\label{crucial lemma for missing part}
Let $E_k$ be an indecomposable rank two summand of $T$. For any
indecomposable vector bundle $X$,
\begin{itemize}
\item[(1)] if $\Ext^1(E_k, X)\neq 0$, then $\Hom(T, X)=0$;
\item[(2)] if $\Ext^1(X, E_k)\neq 0$, then $\Ext^1(T, X)=0$.
\end{itemize}
\end{lemma}

\begin{proof} We only prove the statement (1), since assertion (2)
can be obtained by similar arguments. For contradiction of assertion
(1), we assume there exists an indecomposable summand $T_i$ of $T$
satisfying $\Hom(T_i, X)\neq 0$. Then there exists some $m'\geq 0$,
such that
\[\Hom(\tau^{-m'}T_i, X)\neq 0 \text{\quad and\quad} \Hom(\tau^{-m'-1}T_i, X)=0.
\]
That is, $$\tau^{-m'}T_i\in\cs(\to X).$$ Since $\Ext^1(E_k, X)\neq
0$, by Serre duality, we have $\Hom(\tau^{-1}X, E_k)\neq 0$. Then by
similar arguments, there exists some $m''\geq 1$, such that
$$\tau^{-m''}X\in\cs(\to E_k).$$ It follows that there exists some $m\geq m'+m''\geq
1$, such that $\tau^{-m}T_i\in\cs(\to E_k)$, which implies
$T_i\notin\Dom(E_k)$. Therefore, by Lemma \ref{iff condition of
extension free}, $T_i\oplus E_k$ is not extension-free, a
contradiction. This finishes the proof.
\end{proof}

\begin{corollary}\label{missing part subseteq domain}
For any indecomposable rank two summand  $E_k$ of $T$,
$\ind(\crc)\subseteq\Dom(E_k)$.
\end{corollary}

\begin{proof}
For any indecomposable object $X\in\crc$, by (\ref{form of missing
part}), $\Hom(T, X)\neq 0\neq \Ext^1(T, X)$. Then by Lemma
\ref{crucial lemma for missing part}, we have $\Ext^1(E_k, X)=
0=\Ext^1(X, E_k)$. That is, $E_k\oplus X$ is extension-free. Hence
by Lemma \ref{iff condition of extension free}, $X\in\Dom(E_k)$, as
claimed.
\end{proof}

The following lemma shows that the ``missing part" $\crc$ contains
two components.
\begin{lemma}\label{partition lemma} The ``missing part" $\crc$ has
a decomposition
$$\crc=\crc_1\coprod\crc_2,$$
where$$\ind(\crc_1)=\ind(\crc)\cap\Dom^+(E_i) \text{\quad and\quad}
\ind(\crc_2)=\ind(\crc)\cap\Dom^-(E_j).$$
\end{lemma}

\begin{proof}Firstly, we claim that
\begin{equation}\label{claim in the partition}
\tau^m E_k\notin\crc, \text{\quad for\ any\ } m\in \mathbb{Z}\text{\
and\ } i\leq k \leq j.
\end{equation}
In fact, for $i\leq k \leq j$, if $m=0$, then $\Ext^1(T, E_k)=0$
implies that $E_k\notin\crc$; if $m\neq 0$, then by Corollary
\ref{general homomorphisms related to rank two bundle}, $\tau^m
E_k\notin\Dom(E_k)$. Combining with Lemma \ref{missing part subseteq
domain}, we get $\tau^m E_k\notin\crc$, as claimed. Thus $\crc$ has
a decomposition
$$\crc=\crc_1\coprod\crc_2,$$
where $$ \crc_1=\add\{X\in\ind(\crc)|l(X)<i\}\text{\quad and\quad}
\crc_2=\add\{X\in\ind(\crc)|l(X)>j\}.$$ Moreover, (\ref{claim in the
partition}) implies $E_i\notin\crc$, hence for any object
$X\in\ind(\crc)\cap\Dom^+(E_i)$, we have $l(X)<i$. It follows that
$\ind(\crc)\cap\Dom^+(E_i)\subseteq \ind(\crc_1)$. On the other
hand, for any indecomposable object $X\in \crc_1$, by Corollary
\ref{missing part subseteq domain}, we have $X\in \Dom^+(E_i)$. It
follows that $\ind(\crc_1)\subseteq\ind(\crc)\cap\Dom^+(E_i)$. Hence
$\ind(\crc_1)=\ind(\crc)\cap\Dom^+(E_i)$. Similarly, we have
$\ind(\crc_2)=\ind(\crc)\cap\Dom^-(E_j)$. This finishes the proof.
\end{proof}

We are now going to describe the categories $\crc_1$ and $\crc_2$
more precisely. By symmetric of $\crc_1$ and $\crc_2$, we only show
the results for $\crc_1$ in the following. One should keep in mind
that all the statements also hold for $\crc_2$ by using the duality
of $\cl_1$ (resp. $\cl'_1$) and $\cl_{n+1}$ (resp. $\cl'_{n+1}$).
The following lemma gives an explicit description of $\crc_1$,
comparing with (\ref{form of canonical missing part}).
\begin{lemma}\label{form of up missing part}Assume $\cs(\to
E_i)\cap\cl_1=L$, then
\begin{equation}\label{formular of up missing part }
\crc_1=\add\{X\in\ind(\coh\X)|\Hom(L, X)\neq 0\neq
\Ext^1(L((i-1)\vx_3), X)\}.
\end{equation}
\end{lemma}

\begin{proof}For any indecomposable sheaf $X$, if $\Hom(L, X)\neq 0\neq
\Ext^1(L((i-1)\vx_3), X)$, then $X\in \crc$ since both of $L$ and
$L((i-1)\vx_3)$ are direct summands of $T$. Moreover, by Lemma
\ref{extension-free of domains}, $\Ext^1(L((i-1)\vx_3), X)\neq 0$
implies $X\notin \crc_2$, hence $X\in \crc_1$.

On the other hand, for any indecomposable object $X\in \crc_1$, we
have $X\in\Dom^+(E_k)$ for $i\leq k\leq j$. So by Lemma \ref{iff
condition of extension free}, $\Ext^1(E_k, X)=0$. Moreover, by Lemma
\ref{extension-free of domains}, $\Ext^1(E_j^d, X)=0$. Hence
$\Ext^1(T, X)\neq 0$ implies $\Ext^1(E_i^u, X)\neq 0$. Then there
exists some $0\leq m\leq i-1$, such that $\Ext^1(L(m\vx_3), X)\neq
0$. By applying $\Hom(-, X)$ to the injective morphism
$L(m\vx_3)\hookrightarrow L((i-1)\vx_3)$, we get
\begin{equation}\label{claim in c1}
\Ext^1(L((i-1)\vx_3), X)\neq 0.
\end{equation}
Next, we claim that $\Hom(E_i, X)=0$. Otherwise, we have
$X\in\cs(E_i\to)$. Notice that $L((i-1)\vx_3)\in\cs(E_i\to)$, hence
$X$ and $L((i-1)\vx_3)$ belong to the same slice. It follows that
$X\oplus L((i-1)\vx_3)$ is extension-free, contradicting to
(\ref{claim in c1}), as claimed. So by definition, we get
$\tau^{-1}X\in\Dom^+(E_i)\subseteq\Dom^+(E_j)$. Then by Serre
duality and Lemma \ref{extension-free of domains},
\[\Hom(E_j^d, X)=D\Ext^1(\tau^{-1}X, E_j^d)=0.
\]
Moreover, for any $i\leq k\leq j$,
$\tau^{-1}X\in\Dom^+(E_i)\subseteq\Dom^+(E_k)$ implies that
$\Ext^1(\tau^{-1}X, E_k)=0$. It follows that $\Hom(E_k, X)=0$. Hence
$\Hom(T, X)\neq 0$ implies $\Hom(E_i^u, X)\neq 0$. Then there exists
some $0\leq m\leq i-1$, such that
\begin{equation}\label{non-zero homomorphism of mx3} \Hom(L(m\vx_3),
X)\neq 0.\end{equation} By applying $\Hom(-, X)$ to the exact
sequence $$0\to L\to L(m\vx_3)\to S\to 0,$$ where $S$ is a coherent
sheaf of finite length, we get an exact sequence $$0\to\Hom(S,
X)\to\Hom(L(m\vx_3), X)\to\Hom(L, X).$$ Since $\Hom(S, X)=0$,
(\ref{non-zero homomorphism of mx3}) induces that
\[
\Hom(L, X)\neq 0.
\]
This finishes the proof.
\end{proof}

\begin{lemma}\label{c1 does not contain line bundles}The category $\crc_1$ doesn't contain any
line bundle.
\end{lemma}

\begin{proof}
For contradiction we assume $L(\vx)\in\crc_1$ for some $\vx\in
\mathbb{L}$. Then by (\ref{formular of up missing part }), $$\Hom(L,
L(\vx))\neq 0 \text{\quad and\quad} \Ext^1(L((i-1)\vx_3),
L(\vx))\neq 0.$$ It follows from (\ref{homomorphism between line
bundles}) and (\ref{formula of extension between line bundles}) that
 $$\vx\geq 0\text{\quad and\quad}\vx\leq
\vec{\omega}+(i-1)\vx_3\leq \vec{\omega}+\vc,$$ which is a
contradiction to (\ref{two possibilities of x}), this finishes the
proof.
\end{proof}

Denote by $\mathscr{L}$ the set of all line bundles in $\coh\X$ and
$\ul{\vect}\X=\vect\X/[\mathscr{L}]$ the stable category of vector
bundles obtained by factoring out all line bundles. Then Lemma
\ref{c1 does not contain line bundles} shows that the ``missing
part" $\crc_{1}$ is a subcategory of $\ul{\vect}\X$. For
simplification, we denote $\Hom_{\ul{\vect}\X}(-, -)$ by
$\ul{\Hom}(-, -)$.

\begin{lemma}\label{factor property}
Let $X, Y, Z$ be indecomposable vector bundles satisfying
$$\ul{\Hom}(X, Z)\neq 0\text{\quad and\quad}\ul{\Hom}(Z, Y)\neq 0.$$
Then $X, Y\in\crc_1$ imply $Z\in\crc_1$.
\end{lemma}

\begin{proof}Since $\ul{\Hom}(X, Z)\neq 0$, there exists a non-zero
morphism $\phi: X\to Z$ in $\coh\X$, which can not factor through
line bundles. It follows that $\Im \phi=X$, that is, $\phi$ is
injective. Since $X\in\crc_1$, we have $\Hom(L, X)\neq 0$. Applying
$\Hom(L, -)$ to $\phi$, it follows that $$\Hom(L, Z)\neq 0.$$
Dually, $\ul{\Hom}(Z, Y)\neq 0$ and $Y\in\crc_1$ imply that
$$\Ext^1(L((i-1)\vx_3), Z)\neq 0.$$
Thus by (\ref{formular of up missing part }), $Z\in\crc_1$.
\end{proof}

For any two vector bundles $X, Y$, denote by $[\mathscr{L}](X, Y)$
the ideal of all morphisms from $X$ to $Y$ in $\coh\X$ which factor
through a finite direct sum of line bundles. The following
proposition indicates that $\crc_{1}$ is actually a full subcategory
of $\ul{\vect}\X$.

\begin{proposition}\label{homomorphism in the factor
category} For any two vector bundles $X, Y\in\crc_1$, we have
$$\crc_1(X, Y)=\Hom(X,
Y)/[\mathscr{L}](X, Y).$$
\end{proposition}

\begin{proof}
Obviously, we have a surjection $\pi: \Hom(X,
Y)\twoheadrightarrow\crc_1(X, Y)$. By Lemma \ref{c1 does not contain
line bundles}, we have $[\mathscr{L}](X, Y)\subseteq \Ker\pi$. On
the other hand, for any $f\in\Hom(X,Y)$ with $\pi(f)=0$, we know
that $f$ factors through some object $Z\in\mathcal {X}_{0} \cup
\mathcal {X}_{1}$. We claim that $f$ can factor through a direct sum
of line bundles. Otherwise, there exists an indecomposable summand
$Z_k$ of $Z$ with $\rk Z_k=2$ satisfying
$$\ul{\Hom}(X, Z_k)\neq 0\text{\quad and\quad}\ul{\Hom}(Z_k, Y)\neq 0.$$
Then by Lemma \ref{factor property}, $Z_k\in\crc_1$, a
contradiction. It follows that $f\in[\mathscr{L}](X, Y)$. Hence
$\Ker\pi=[\mathscr{L}](X, Y)$, this finishes the proof.
\end{proof}

\begin{lemma}\label{abelian of c1} The category $\crc_1$ is an
abelian category.
\end{lemma}

\begin{proof}For any indecomposable object $X\in\crc_1$, applying $\Hom(-, X)$
to the injective map $L((i-1)\vx_3)\hookrightarrow L(\vc)$, we get
$\Ext^1(L(\vc), X)\neq 0$. Hence $X\in\crc_{L}$. That is, $\crc_1$
is a subcategory of $\crc_{L}$. Moreover, for any vector bundles $X,
Y\in \crc_{1}$, by the same proof of Proposition \ref{homomorphism in the factor
category}, we know that $$\crc_L(X, Y)=\Hom(X, Y)/[\mathscr{L}](X,
Y)=\crc_{1}(X, Y).$$ Hence $\crc_1$ is a full subcategory of
$\crc_{L}$.

We are going to show that $\crc_1$ is closed under extension, kernel
of surjective and cokernel of injective in $\crc_{L}$. Then $\crc_1$
inherits the abelianness of $\crc_{L}$.

Firstly, we show that $\crc_1$ is closed under extension in
$\crc_L$. Let
$$0\to X\to \oplus_{k}Z_k \to Y\to 0$$ be a non-split exact sequence
in $\crc_{L}$ with $X, Y$ from $\crc_1$. Without loss generality, we
assume $X, Y$ are indecomposable. We need to show each
$Z_k\in\crc_1$. It suffices to show that
\begin{equation}\label{claim for extension closed}
\Ext^1(L((i-1)\vx_3), Z_k)\neq 0.
\end{equation}
There are two cases to consider:
\begin{itemize}
\item[\emph{Case} 1:]$\crc_{L}(Z_k, Y)\neq 0$, then $\ul{\Hom}(Z_k, Y)\neq 0$ follows. By the proof of
Lemma \ref{factor property}, each non-zero morphism $\phi: Z_k\to Y$
is injective in $\coh\X$. Assume $\phi$ fits into the following
exact sequence in $\coh\X$:
\begin{equation}\label{exact sequence for fitting of phi}
0\to Z_k\xrightarrow{\phi} Y\to S\to 0,
\end{equation}where $S$ is a coherent sheaf of finite length since both $Z_k$ and $Y$ are of rank two.
Notice that $Y\in\crc_1$, it follows that
\begin{equation}\label{non-zero extension of y}
\Ext^1(L((i-1)\vx_3), Y)\neq 0.
\end{equation}Now by applying
$\Hom(L((i-1)\vx_3), -)$ to (\ref{exact sequence for fitting of
phi}), we get an exact sequence $\Ext^1(L((i-1)\vx_3),
Z_k)\to\Ext^1(L((i-1)\vx_3), Y)\to\Ext^1(L((i-1)\vx_3), S).$ Note
that $\Ext^1(L((i-1)\vx_3), S)=0$. Then (\ref{claim for extension
closed}) follows from (\ref{non-zero extension of y}).

\item[\emph{Case} 2:]$\crc_{L}(Z_k, Y)= 0$, then $X\to Z_k$ is
surjective in $\crc_{L}$. Notice from \cite{CLR} that $\crc_{L}$ is
equivalent to the module category of $k\overrightarrow{A_{n-1}}$.
Hence from the structure of $\crc_L$, we know that in
$\Gamma(\vect\X)$,
\begin{equation}\label{property of zk}
Z_k\in\cs(X\to) \text{\quad and\quad} l(Z_k)\leq l(X).
\end{equation}
Since $X\in\crc_1$, it follows that $\Ext^1(L((i-1)\vx_3), X)\neq
0$. Then by Serre duality,
\begin{equation}\label{non-zero homomorphism of x}
\Hom(X, \tau L((i-1)\vx_3))\neq 0.
\end{equation}
Hence there exists some $m\geq 1$, such that $\tau^{m}
L((i-1)\vx_3)\in\cs(X\to)$. Combining with (\ref{property of zk}),
and noticing that $\tau^{m} L((i-1)\vx_3)$ is of length one, we get
$\tau^{m} L((i-1)\vx_3)\in\cs(Z_k\to)$. Hence $\Hom(Z_k, \tau^{m}
L((i-1)\vx_3))\neq 0.$ According to Corollary \ref{general
homomorphisms related to rank two bundle},
$$\Hom(Z_k, \tau L((i-1)\vx_3))\neq 0.$$ By Serre duality, (\ref{claim
for extension closed}) holds, as claimed.
\end{itemize}

Secondly, assume $\phi: X\to Y$ is surjective in $\crc_L$ with $X,
Y\in\crc_1$. If $\phi$ is isomorphism, then $\Ker\phi=0\in\crc_1$.
Otherwise, for any indecomposable direct summand $Z_k$ of
$\Ker\phi$, there exists an non-zero morphism from $Z_k$ to $X$ in
$\crc_L$. By the same proof of $\emph{Case}\ 1$ above, $X\in\crc_1$
implies that $Z_k\in\crc_1$.

Thirdly, assume $\phi: X\to Y$ is injective in $\crc_L$ with $X,
Y\in\crc_1$. If $\phi$ is isomorphism, then $\Coker\phi=0\in\crc_1$.
Otherwise, for any indecomposable direct summand $Z_k$ of
$\Coker\phi$, $Y\to Z_k$ is surjective. 
By similar arguments of $\emph{Case}\ 2$ above, $Y\in\crc_1$ implies
that $Z_k\in\crc_1$. This finishes the proof.
\end{proof}

Similarly, we have
\begin{lemma}\label{abelian of c2} Assume $\cs(E_j\to)\cap\cl_{n+1}=L'$, then
$$\crc_2=\add\{X\in\ind(\coh\X)|\Hom(L'(-(n-j+1)\vx_3), X)\neq 0\neq
\Ext^1(L', X)\}.$$ Moreover, $\crc_2$ is an abelian category.
\end{lemma}

Now we give our main result in this section.
\begin{theorem}\label{main theorem}
Let $T$ be a tilting bundle of the form (\ref{form of T}) and $\crc$
be the corresponding ``missing part". Assume $\cs(\to
E_i)\cap\cl_{1}=L$ and $\cs(E_j\to)\cap\cl_{n+1}=L'$. Then $\crc$
can be decomposed to a product of two abelian categories
$$\crc=\crc_1\coprod\crc_2,$$
where
$$\crc_1=\add\{X\in\ind(\coh\X)|\Hom(L, X)\neq 0\neq
\Ext^1(L((i-1)\vx_3), X)\},$$ and
$$\crc_2=\add\{X\in\ind(\coh\X)|\Hom(L'(-(n-j+1)\vx_3), X)\neq 0\neq
\Ext^1(L', X)\}.$$
\end{theorem}

\begin{proof}Combine with Lemma \ref{partition lemma}, Lemma \ref{form of up missing part}, Lemma \ref{abelian of c1} and Lemma \ref{abelian of
c2} to finish the proof.
\end{proof}

For other choices of $T^+(E_i)\in \Br^+(E_i)$ and $T^-(E_j)\in
\Br^-(E_j)$, by similar arguments one can obtain the similar
decomposition of $\crc$ with some modifications of $L$ and $L'$.
Summarize up, we have
\begin{theorem}\label{main general theorem}
Let $T$ be a tilting bundle of the form (\ref{form of tilting
bundles}) and $\crc$ be the corresponding ``missing part". Then
$\crc$ can be decomposed to a product of two abelian categories.
\end{theorem}

\bibliographystyle{amsplain}

\end{document}